\documentclass{amsart}

\newif\ifLONGVER
\LONGVERtrue

\newif\ifULTRAMODE
\ULTRAMODEtrue

\usepackage{CJKutf8}
\usepackage[utf8]{inputenc} 
\usepackage{amsthm}
\usepackage{amsfonts}
\usepackage{scalerel}
\usepackage{amssymb}
\usepackage{ mathrsfs }
\usepackage{mathtools}
\usepackage{xcolor}
\usepackage{colonequals}
\usepackage{xfrac}
\usepackage{caption}
\usepackage{environ}
\usepackage{tikz}
\usepackage{adjustbox}
\usepackage{booktabs,tabularx}
\usepackage[normalem]{ulem}
\usepackage[shortlabels]{enumitem}
\usepackage{setspace}
\usepackage[hyphens]{url}
\usepackage{hyperref}
\usepackage[hyphenbreaks]{breakurl}
\usepackage[noabbrev]{cleveref}

\usepackage[citestyle=alphabetic,
style=alphabetic,
giveninits=true,
natbib=true,
backend=biber,
]
{biblatex}

\DeclareSourcemap{
  \maps[datatype=bibtex]{
    \map[overwrite=true]{
      \step[fieldset=url, null]
	  \step[fieldset=doi, null]
	  \step[fieldset=isbn, null]
	  \step[fieldset=issn, null]
    }
  }
}

\newcommand{\lang}[2]{\mscr{L}_{{\scriptscriptstyle #1}}^{{\scriptscriptstyle #2}}}

\newcommand{\bb}{\mathbb}

\newcommand{\mcal}{\mathcal}
\newcommand{\mscr}{\mathscr}

\newcommand{\mfrak}{\mathfrak}

\newcommand{\pow}{\mcal{P}}

\newcommand{\wo}{\setminus}

\newcommand{\upperRomannumeral}[1]{\uppercase\expandafter{\romannumeral#1}}

\makeatletter

\newcommand{\Rmn}[1]{\expandafter\@slowromancap\romannumeral #1@}
\makeatother
\makeatletter
\newcommand{\oset}[3][0ex]{%
  \mathrel{\mathop{#3}\limits^{
    \vbox to#1{\kern-2\ex@
    \hbox{$\scriptstyle#2$}\vss}}}}
\makeatother

\DeclareMathOperator{\suf}{s}
\DeclareMathOperator{\prf}{p}
\newcommand{\ps}{\suf^{-1}}
\newcommand{\pp}{\prf^{-1}}
\newcommand{\ips}{\mfrak{s}^{-1}}

\newcommand{\lcro}{\mathbb{I}}

\newcommand{\langWI}{\lang{\WSO(\lcro)}{}}
\newcommand{\langLI}{\lang{\LCI(\lcro)}{}}

\newcommand{\bdr}[1]{#1^{\partial}}
\newcommand{\cz}{\mathbf{0}}

\DeclareMathOperator{\opp}{op}

\DeclareMathOperator{\At}{At}

\DeclareMathOperator{\fin}{fin}

\DeclareMathOperator{\WSO}{W}

\DeclareMathOperator{\LCI}{L}

\DeclareMathOperator{\bdd}{Bd}

\DeclareMathOperator{\fci}{fci}

\newtheorem{theorem}{Theorem}[section]
\newtheorem{lemma}[theorem]{Lemma}
\newtheorem{proposition}[theorem]{Proposition}
\newtheorem{corollary}[theorem]{Corollary}

\theoremstyle{definition}
\newtheorem{definition}[theorem]{Definition}

\theoremstyle{remark}
\newtheorem{remark}[theorem]{Remark}
\newtheorem{notation}[theorem]{Notation}

\numberwithin{equation}{section}

\setlist[enumerate,itemize]{leftmargin=0.75cm,itemsep=1.5pt,topsep=1.5pt}

\begin{document}

\title[]{Model-completeness for the lattice of finite unions of closed intervals of a dense linear order}

\author[]{Deacon Linkhorn}
\address{The University of Manchester, School of Mathematics, Oxford Road, Manchester, M13 9PL, UK}
\curraddr{}
\email{deacon.linkhorn@manchester.ac.uk}
\thanks{}



\date{\today}

\dedicatory{}



\maketitle

\vspace{-0.5cm}

\tableofcontents

\vspace{-0.5cm}

\section{Introduction}

Let $\lcro$ be any dense linear order with left endpoint but no right endpoint.
We consider $\pow_{\fci}(\lcro)$, the collection of finite unions of closed intervals of $\lcro$. 
This collection arises naturally in the setting of o-minimality, as precisely the lattice of closed definable sets in any o-minimal expansion of $\lcro$.
Our main result is \cref{Thm-LImc} which says that $\LCI(\lcro)$, the expansion of the lattice $(\pow_{\fci}(\lcro),\cup,\cap)$ by constants for the empty set and $\{0\}$ (here $0$ is the left endpoint of $\lcro$) as well as four unary functions, is model-complete.

The proof makes use of previous results regarding the weak monadic second order theory of $\lcro$ from the authors PhD thesis \cite{DLPhD}. 

\section{Definitions and setup}

Throughout the note we fix a dense linear order $\lcro$, which has a left endpoint but no right endpoint.
By the fact that any completion of the theory of dense linear orders is $\aleph_0$-categorical, and using Proposition 2.7 from \cite{TresslTop}, it follows that our results do not depend the particular choice of $\lcro$.

Our logical notation, where not explicitly defined, is taken from \cite{HodgesBible}. 

\begin{definition}
Let $S$ be a set.
We write $\pow(S)$ for the powerset of $S$.
We write $\pow_{\fin}(S)$ for the collection of finite subsets of $S$.
\end{definition}

\begin{definition}
Let $\alpha$ be any linear order.
A \textbf{closed interval} of $\alpha$ is a subset $S \subseteq \alpha$ of one of the following forms,
\begin{enumerate}
\item $[i,j] \coloneqq \{k \in \alpha: i \leq k \leq j\}$ for $i,j \in \alpha$,
\item $[i,+\infty) \coloneqq \{k \in \alpha: i \leq k\}$ for $i \in \alpha$,
\item $(-\infty,j] \coloneqq \{k \in \alpha: k \leq j\}$ for $j \in \alpha$.
\end{enumerate}
We will write $\pow_{\fci}(\alpha)$ for the set of finite unions of closed intervals of $\alpha$.
\end{definition}

Note that $\pow_{\fin}(\alpha) \subseteq \pow_{\fci}(\alpha) \subseteq \pow(\alpha)$.

\begin{definition}
Let $\alpha$ be a discrete linear order with endpoints (i.e. $\alpha$ is finite, or has order type $\bb{N} + \Gamma \cdot \bb{Z} + \bb{N}^{\opp}$ for some linear order type $\Gamma$).

We will write $\suf_{\alpha}$ for the successor function $\alpha \wo \max(\alpha) \rightarrow \alpha$ and $\prf_{\alpha}$ for the predecessor function $\alpha \wo \min(\alpha) \rightarrow \alpha$.

We will moreover write $\ps_{\alpha}$ for the function $\pow(\alpha) \rightarrow \pow(\alpha)$ sending $A \subseteq \alpha$ to $\{i \in \alpha: \suf_{\alpha}(i) \in A\}$, and $\pp_{\alpha}$ for the function $\pow(\alpha) \rightarrow \pow(\alpha)$ sending $A \subseteq \alpha$ to $\{i \in \alpha: \prf_{\alpha}(i) \in A\}$. 
\end{definition}

\begin{definition}
Let $\alpha$ be any linear order. 
We call $A \in \pow(\alpha)$ \textbf{discrete} if the restriction of $\alpha$ to $A$ is a discrete linear order.

For $A \subseteq \alpha$ discrete we will write $\suf_A$ (respectively $\prf_A$) for the successor (respectively predecessor) function on the restriction of $\alpha$ to a linear ordering on $A$.
Note that every finite subset of $\alpha$ is discrete. 
\end{definition}

\section{\texorpdfstring{$\WSO(\lcro)$}{W(I)}, the weak monadic second order version of \texorpdfstring{$\lcro$}{I}}

Recall that $\lcro$ is a dense linear order with left endpoint but no right endpoint that we fixed at the outset.

\begin{definition}
We write $0$ for the left endpoint of $\lcro$, i.e. the smallest element with respect to the ordering.
We write $\langWI$ for the signature $\{\cup,\cap,\bot,\cz,\min,\max,\ips\}$.\footnote{This comprises in order, two binary function symbols, two constants, two unary function symbols, and a binary function symbol.}
Then $\WSO(\lcro)$ is the $\langWI$-structure with,
\begin{enumerate}
\item universe $\pow_{\fin}(\lcro)$, the collection of finite subsets of $\lcro$,
\item $\cup,\cap$ interpreted as the operations of union and intersection, 
\item $\bot$ interpreted as the empty set,
\item $\cz$ interpreted as $\{0\}$,
\item $\min$ and $\max$ interpreted as the operations taking a non-empty finite set to the singleton containing its minimum and maximum respectively, with respect to the ordering of $\lcro$ (and both fixing $\bot$),
\item $\ips$ being the binary function given by,
\[
\ips(A,B) = \{i \in A:\suf_A(i) \in B\},
\]
i.e. $\ips$ is a single binary function which encodes the (preimage map associated to) the family of successor functions on finite subsets of $\lcro$.
\end{enumerate}
\end{definition}

\begin{notation}
As is conventional, we will sometimes use $X \subseteq Y$ as a shorthand for the formula $X \cap Y = X$.
Note that in doing so we conceal an unnested atomic formula, so this shorthand will never lead us to mistake an existential formula for a quantifier-free formula.
\end{notation}

\begin{lemma}\label{Lemma-notbotelim}
The following are equivalent for each $A \in \pow_{\fin}(\lcro)$,
\begin{enumerate}
\item $A \neq \bot$,
\item $A = \cz$ or $\cz \subseteq \ips(A \cup \cz,A)$.
\end{enumerate}
\end{lemma}
\begin{proof}
$(1) \Rightarrow (2)$:
If $A \neq \bot$ and $A \neq \cz$ then it is immediate that $A \wo \cz$ is not $\bot$.
This gives us that $\suf_{A \cup \cz}(0)$ is defined (i.e. $0 \in \lcro$ has a successor in $\cz \cup A$).
Then $\suf_{A \cup \cz}(0) \in A$ which by definition gives us $\cz \subseteq \ips(A \cup \cz,A)$.

$(2) \Rightarrow (1)$:
Suppose $A = \bot$. 
From this it follows immediately that $A \neq \cz$.
Moreover $\ips(A \cup \cz,A) = \ips(\cz,\bot)$ which is $\bot$, hence $\cz \not\subseteq \ips(A \cup \cz,A)$.
\end{proof}

\begin{definition}
Recall that for sets $A$ and $B$, the \textbf{relative complement} of $B$ inside $A$, denoted by $A \wo B$, is the set $\{i \in A: i \notin B\}$. 
\end{definition}

\begin{proposition}\label{Prop-WIqffelimneg}
Every quantifier-free $\langWI$-formula is equivalent to a positive (i.e. negation free) existential $\langWI$-formula over the the theory of $\WSO(\lcro)$. 
\end{proposition}
\begin{proof}
First note that for each $A,B \in \pow_{\fin}(\lcro)$, $A \wo B = C$ if and only if,
\[ 
(A \cap  B) \cup C = A \text{ and }B \cap C = \bot.
\]
Therefore every positive existential $(\langWI \cup \{\wo\})$-formula is equivalent to a positive existential $\langWI$-formula (working in $\WSO(\lcro)$, with $\wo$ a binary function symbol interpreted as relative complement).
So it is enough to check that every quantifier-free $\langWI$-formula is equivalent to a positive existential $(\langWI \cup \{\wo\})$-formula.
It is enough to check that the negation of an atomic $(\langWI \cup \{\wo\})$-formula is equivalent to a positive existential $(\langWI \cup \{\wo\})$-formula in $\WSO(\lcro)$.
Atomic formulas are all of the form $q_1(\bar{X}) = q_2(\bar{X})$ for $\langWI$-terms $q_1,q_2$.
Let $\Delta(q_1(\bar{X}),q_2(\bar{X}))$ be shorthand for $(q_1(\bar{X}) \wo q_2(\bar{X})) \cup (q_2(\bar{X}) \wo q_1(\bar{X}))$, i.e. the symmetric difference of $q_1$ and $q_2$.
The formula $\neg (q_1(\bar{X}) = q_2(\bar{X}))$ is equivalent over $\WSO(\lcro)$ to,
\[
\exists Y (Y \neq \bot \wedge Y \subseteq \Delta(q_1(\bar{X}),q_2(\bar{X}))),
\]
but using \cref{Lemma-notbotelim} this in turn is equivalent over $\WSO(\lcro)$ to,
\[
\exists Y ((Y = \cz \vee \cz \subseteq \ips(Y \cup \cz,Y)) \text{ and }Y \subseteq \Delta(q_1(\bar{X}),q_2(\bar{X}))),
\]
which is positive existential as required.
\end{proof}

\begin{theorem}\label{Thm-WIposmodelcomplete}
$\WSO(\lcro)$ is positive-model-complete in the signature $\langWI$.
\end{theorem}
\begin{proof}
That $\WSO(\lcro)$ is model-complete is proved in my thesis \cite{DLPhD} (see theorem 5.4.7 on page 141).
The result then follows immediately from \cref{Prop-WIqffelimneg}.
\end{proof}

\section{\texorpdfstring{$\LCI(\lcro)$}{L(I)}, finite unions of closed intervals of \texorpdfstring{$\lcro$}{I}}

\begin{definition}
We write $\langLI$ for the signature $\{\cup,\cap,\bot,\cz,\min,\max,l,r\}$.\footnote{This comprises in order, two binary function symbols, two constants, and four unary function symbols.}

$\LCI(\lcro)$ is the $\langLI$-structure with,
\begin{enumerate}
\item universe $\pow_{\fci}(\lcro)$, the collection of finite unions of closed intervals of $\lcro$,
\item $\cup,\cap$ interpreted as the operations of union and intersection, 
\item $\bot$ interpreted as the empty set,
\item $\cz$ interpreted as $\{0\}$,
\item $\min$ and $\max$ interpreted as the operations taking an element of $\pow_{\fci}(\lcro)$ to the singleton containing its minimum and maximum respectively, with respect to the ordering of $\lcro$ (we set $\min(\bot) = \max(\bot) = \bot$, and in the case $A \in \pow_{\fci}(\lcro)$ is unbounded we set $\max(A)=\bot$),
\item $l$ and $r$ are interpreted as the operations taking an element of $\pow_{\fci}(\lcro)$ to the set of its left and right endpoints respectively.
\end{enumerate}
\end{definition}

\begin{notation}\label{notationlr}
For $A \in \pow_{\fci}(\lcro)$ we will sometimes write $A_l$ and $A_r$ in place of $l(A)$ and $r(A)$ respectively. We will write $\bdr{A}$ as shorthand for $A_l \cup A_r$.
\end{notation}

\begin{lemma}\label{Lemma-boundeddef}
The bounded elements of $\pow_{\fci}(\lcro)$ form a definable subset in  $\LCI(\lcro)$, which we will denote by $\bdd$.
\end{lemma}
\begin{proof}
Under our interpretation, an element $A \in \pow_{\fci}(\lcro)$ is bounded if and only if $A = \bot$ or $\max(A) \neq \bot$. 
Conversely the unbounded elements are precisely those $A \in \pow_{\fci}(\lcro)$ for which $A \neq \bot$ and $\max(A) = \bot$.

Therefore $\bdd$ is in fact a quantifier-free definable set in $\LCI(\lcro)$. 
\end{proof}

\section{Interpreting \texorpdfstring{$\WSO(\lcro)$ in $\LCI(\lcro)$}{W(I) in L(I)}}

\begin{proposition}
The set $\pow_{\fin}(\lcro)$ is quantifier-free definable in $\LCI(\lcro)$.
\end{proposition}
\begin{proof}
The formula $l(X) = r(X)$ defines $\pow_{\fin}(\lcro)$ in $\LCI(\lcro)$.
\end{proof}

We will use this as the foundation for our interpretation of $\WSO(\lcro)$ in $\LCI(\lcro)$.

In the remainder of this section we will outline how to define the remaining $\langWI$-structure carried by $\WSO(\lcro)$ within $\LCI(\lcro)$ on the set $\pow_{\fin}(\lcro)$.
 
It will be important for us, when transferring model-completeness from $\WSO(\lcro)$ to $\LCI(\lcro)$ in \cref{Section-transfer}, that existential $\langLI$-formulae are used to do so.

\begin{remark}
If $\phi$ is an unnested atomic formulae in the signature $\langWI \cap \langLI$, i.e. $\{\cup,\cap,\bot,\cz,\min,\max\}$, we have that,
\[
\phi(\WSO(\lcro)) = \phi(\LCI(\lcro)) \cap \pow_{\fin}(\lcro).
\]
As such for unnested atomic formulae in this reduct, we need do nothing when giving the interpretation. 
Here we use unnested atomic $\langLI$-formulas, which a fortiori are existential. 
\end{remark}

All that is left is to produce an $\langLI$-formula which defines $\ips$ in $\LCI(\lcro)$.

\begin{remark}
Let $A,B \in \pow_{\fin}(\lcro)$. 
If $A = \bot$ or $B = \bot$ then $\ips(A,B) = \bot$.
Therefore in defining $\ips$ we can assume that $A,B \neq \bot$.

Moreover we have that $\ips(A,B) = \ips(A,B \cap A)$. 
From this it is clear that it is sufficient to define the relation $\ips(A,B) = C$ in the case where $B \subseteq A$. 
\end{remark}

\begin{proposition}
Let $A,B,C \in \pow_{\fin}(\lcro)$ with $\bot \subsetneq B \subseteq A$. Then $\ips(A,B) = C$ if and only if there exists $D \in \pow_{\fci}(\lcro)$ such that one of the following holds,
\begin{enumerate}
\item $\min(A) \subseteq B$, $l(D) = (B \wo \min(B)) \cup \cz$, $r(D) = C$, and $C \subseteq A \subseteq D$, or,
\item $\min(A) \not\subseteq B$, $l(D) = B \cup \cz$, $r(D) = C$, and $C \subseteq A \subseteq D$.
\end{enumerate} 
\end{proposition}
\begin{proof}
First suppose that $\ips(A,B)=C$, we will show that one of the two conditions must then hold.
Let $E$ be the union of all of the \emph{open} intervals of $\lcro$ of the form $(i,j)$ for some $i \in C$ and $j \in B$ such that $\suf_{A}(i) = j$. 
Now take $D = \lcro \wo E \in \pow_{\fci}(\lcro)$.

By definition we get that $C \subseteq A$. Then $A \subseteq D$ follows from the choice that $E$ be made up from intervals $(i,j)$ where $\suf_A(i) = j$.
For if $A \not\subseteq D$ then we get $k \in A$ such that for some $i,j$ with $\suf_A(i)=j$ we have $i < k < j$, a contradiction.

For our choice of $E$, it is easy to check that $l(E) = C$ and $r(E)$ is either,
\begin{enumerate}[(i)]
\item $B \wo \min(B)$ if $\min(A) \subseteq B$, or,
\item $B$ if $\min(A) \not\subseteq B$.
\end{enumerate}
Taking the complement, we interchange left and right endpoints, and introduce $0$ as a left endpoint.
This gives us precisely that $(1)$ or $(2)$ hold for our choice of $D$.

Now for the other direction, suppose that $D \in \pow_{\fci}(\lcro)$ exists such that $(2)$ holds. 
We want to show that $\ips(A,B) = C$. 
Let $i \in \ips(A,B)$, so $i \in A$ is such that $\suf_A(i) \in B \subseteq l(D)$. 
Now $i \notin C = r(D)$ implies that $i \in A \wo D$, contradicting our assumption that $A \subseteq D$.
So we have established that $\ips(A,B) \subseteq C$ follows from $(2)$.
Let $i \in C$, so by $(2)$ we have $i \in r(D)$. 
Suppose towards a contradiction that $\suf_A(i) \notin B$. 
Then moreover $\suf_i(A) \notin B \cup \cz = l(D)$.
This again gives us that $i \in A \wo D$, contradicting our assumption that $A \subseteq D$. 
Therefore we have shown that $\ips(A,B)=C$ follows from $(2)$.

We leave the checking of details in showing that $\ips(A,B)=C$ follows from $(1)$ to the reader.
\end{proof}

\begin{corollary}
There is an existential $\langLI$-formula $\phi(X,Y,Z)$ such that for all $A,B,C \in \pow_{\fin}(\lcro)$, $\LCI(\lcro) \models \phi(A,B,C)$ if and only if $\ips(A,B) = C$.
\end{corollary}

\begin{corollary}\label{LinterpretsW}
For each $\langWI$-formula $\phi(\bar{X})$ there is an $\langLI$-formula $\psi(\bar{Y})$ (with $\bar{X}$ and $\bar{Y}$ having the same length) such that for each $\bar{A} \in \pow_{\fin}(\lcro)$,
\[
\WSO(\lcro) \models \phi(\bar{A}) \Longleftrightarrow \LCI(\lcro) \models \psi(\bar{A}).
\]
Moreover $\psi(\bar{Y})$ can always be chosen to be an existential $\langLI$-formula.
\end{corollary}

\section{Interpreting \texorpdfstring{$\LCI(\lcro)$ in $\WSO(\lcro)$}{L(I) in W(I)}}

To make things more easily digestible, we will use $(\pow_{\fci}(\lcro),\subseteq)$ as an intermediary between $\LCI(\lcro)$ and $\WSO(\lcro)$.

It is straightforward, using modified versions of arguments from \cite{TresslTop}, to show that $(\pow_{\fci}(\lcro),\subseteq)$ and $\LCI(\lcro)$ have the same definable subsets. 
Using this result, together with the interpretation of $(\pow_{\fci}(\lcro),\subseteq)$ in $\WSO(\lcro)$ which we are about to give, we will indicate a particular interpretation of $\LCI(\lcro)$ in $\WSO(\lcro)$.

\begin{lemma}
Let $B,C \in \pow_{\fin}(\lcro)$. 
The following are equivalent,
\begin{enumerate}
\item there is $A \in \pow_{\fci}(\lcro) \wo \{\bot\}$ such that $A_l = B$ and $A_r = C$,
\item $B \neq \bot$, $\min(B \cup C) \subseteq B$, and one of the following holds,
\begin{enumerate}
\item $\max(B \cup C) \subseteq C \text{ and } \ips(B \cup C,C \wo B) = B \wo C$,
\item $\max(B \cup C) \subseteq B \wo C \text{ and } \ips(B \cup C,C \wo B) \cup \max(B \cup C) = B \wo C$.
\end{enumerate}
\end{enumerate}
\end{lemma}
\begin{proof}
Suppose $(1)$ holds.
Then we can rewrite $(2)$ as follows,
\begin{enumerate}[(2)]
\item $A_l \neq \bot$, $\min(A_l \cup A_r) \subseteq A_l$, and one of the following holds,
\begin{enumerate}
\item $\max(A_l \cup A_r) \subseteq A_r \text{ and } \ips(A_l \cup A_r,A_r \wo A_l) = A_l \wo A_r$,
\item $\max(A_l \cup A_r) \subseteq A_l \wo A_r \text{ and } \ips(A_l \cup A_r,A_r \wo A_l) \cup \max(A_l \cup A_r) = A_l \wo A_r$.
\end{enumerate}
\end{enumerate}
Intuitively then, (2) first says that $A$ has at least one left endpoint, and that the smallest endpoint of $A$ is a left endpoint. Then both (a) and (b) simply say that the proper left endpoints and proper right endpoints appear in pairs, with proper right endpoints immediately preceded by proper left endpoints.
An exception is needed simply for the case where $A$ is unbounded, in which case the largest endpoint is a proper left endpoint which is not the predecessor of a proper right endpoint (this is dealt with by (b)).

Conversely, suppose that $(2)$ holds. 
It is straightforward to construct $A \in \pow_{\fci}(\lcro)$ such that $A_l = B$ and $A_r = C$. 
\end{proof}

This lemma gives us the universe for our interpretation of $(\pow_{\fci}(\lcro),\subseteq)$ in $\WSO(\lcro)$.
We will identify an element $A \in \pow_{\fci}(\lcro)$ with the pair $(A_l,A_r) \in \pow_{\fin}(\lcro)$.
The lemma tells us precisely that the image of the map $\pow_{\fci}(\lcro) \rightarrow \pow_{\fin}(\lcro)^2$ given by $A \mapsto (A_l,A_r)$ is definable in $\WSO(\lcro)$. 
As moreover this map is injective, our interpretation can make use of the equality in $\WSO(\lcro)$ to interpret equality from $\LCI(\lcro)$, in particular we do not need to take a quotient of the image by a definable equivalence relation.

It remains to show that the relation $\subseteq$ on $\pow_{\fci}(\lcro)$ is interpretable in $\WSO(\lcro)$.

\begin{lemma}
There is an $\langWI$-formula $\phi_{\in}(X_l,X_r,Z)$ such that for any $i \in \lcro$ and $A \in \pow_{\fci}(\lcro)$, $i \in A$ if and only if,
\[
\WSO(\lcro) \models \phi_{\in}(A_l,A_r,\{i\}).
\]
\end{lemma}
\begin{proof}
We split into two cases, according to whether $A \in \pow_{\fci}(\lcro)$ is bounded or unbounded (see \cref{Lemma-boundeddef}).
Let $\phi_{\bdd}(X_l,X_r,Z)$ be the $\langWI$-formula,
\[
\ips(\bdr{X} \cup Z,Z) \subseteq X_l \wo X_r \text{ and } \ips(\bdr{X} \cup Z,X_r \wo X_l) \cap Z \neq \bot,
\]
this `says' that the predecessor of $Z$ is a proper left endpoint and that $Z$ is the predecessor of a right endpoint,
Then let $\phi_{\neg \bdd}(X_l,X_r,Z)$ be the $\langWI$-formula,
\[
(\ips(\bdr{X} \cup Z,Z) \subseteq X_l \wo X_r \text{ and } \ips(\bdr{X} \cup Z,X_r \wo X_l) \cap Z \neq \bot) \text{ or } Z = \max(\bdr{X} \cup Z),
\]
this `says' that the condition in $\phi_{\bdd}$ holds or that $Z$ is greater than or equal to all left and right endpoints. 
Then for each $i \in \lcro$ and $A \in \pow_{\fci}(\lcro)$ the following are equivalent,
\begin{enumerate}
\item $i \in A$,
\item $\WSO(\lcro) \models \bdr{A} \neq \bot$ (so that $A \neq \bot$) and one of the following holds,
\begin{enumerate}
\item $\WSO(\lcro) \models \{i\} \subseteq \bdr{A}$ or,
\item $\WSO(\lcro) \models \bdd(A)$ and $\phi_{\bdd}(A_l,A_r,\{i\})$, or,
\item $\WSO(\lcro) \models \neg \bdd(A)$ and $\phi_{\neg \bdd}(A_l,A_r,\{i\})$.
\end{enumerate}
\end{enumerate}
In other words $i \in A$ if and only if $A$ is nonempty and either $i$ is an endpoint of $A$, or $i$ sits between a left and right endpoint of $A$, or $A$ is unbounded and $i$ sits in the unbounded part of $A$.
The latter we have shown is definable in $\WSO(\lcro)$ by giving $\langWI$-formulas, so $\phi_{\in}$ exists.
\end{proof}

\begin{proposition}
There is an $\langWI$-formula $\phi_{\subseteq}(X_l,X_r,Y_l,Y_r)$ such that for any $A,B \in \pow_{\fci}(\lcro)$, $A \subseteq B$ if and only if,
\[
\WSO(\lcro) \models \phi_{\subseteq}(A_l,A_r,B_l,B_r).
\]
\end{proposition}
\begin{proof}
Let $\At(Z)$ be the formula $Z \neq \bot \wedge Z = \min(Z)$.
In $\WSO(\lcro)$ this defines the collection of singletons. 
Therefore we can take for $\phi_{\subseteq}$ the $\langWI$-formula,
\[
\forall Z (\At(Z) \rightarrow (\phi_{\in}(X_l,X_r,Z) \rightarrow \phi_{\in}(Y_l,Y_r,Z))).
\]
This says that every singleton contained in $X$ is contained in $Y$, as required.
\end{proof}

\begin{theorem}
There is an interpretation of $\LCI(\lcro)$ in $\WSO(\lcro)$, for which the co-ordinate map $\pow_{\fci}(\lcro) \rightarrow \pow_{\fin}(\lcro)^2$ is given by $A \mapsto (A_l,A_r)$.
\end{theorem}

\begin{corollary}\label{WinterpretsL}
For each $\langLI$-formula $\phi(\bar{X})$ there is an $\langWI$-formula $\psi(\bar{Y})$ such that for each $\bar{A} \in \pow_{\fci}(\lcro)$,
\[
\LCI(\lcro) \models \phi(\bar{A}) \Longleftrightarrow \WSO(\lcro) \models \psi(\bar{A_l},\bar{A_r}).
\]
\end{corollary}

\section{Transfer of model-completeness from \texorpdfstring{$\WSO(\lcro)$ to $\LCI(\lcro)$}{W(I) to L(I)}}\label{Section-transfer}

\begin{theorem}\label{Thm-LImc}
The $\langLI$-structure $\LCI(\lcro)$ is model-complete.
\end{theorem}
\begin{proof}
Let $\phi(\bar{X})$ be an $\langLI$-formula.
We will show that over $\LCI(\lcro)$ the formula $\phi$ is equivalent to an existential formula $\phi^*$.

Using our interpretation of $\LCI(\lcro)$ in $\WSO(\lcro)$ (in particular \cref{WinterpretsL}), for each $\langLI$-formula $\phi(\bar{X})$ there is an $\langWI$-formula $\psi(\bar{Y})$ such that for each $\bar{A} \in \pow_{\fci}(\lcro)$,
\[
\LCI(\lcro) \models \phi(\bar{A}) \Longleftrightarrow \WSO(\lcro) \models \psi(\bar{A_l},\bar{A_r}). \tag{$\star$}\label{star}
\]
Without loss of generality we can take $\psi(\bar{Y})$ to be a positive existential $\langWI$-formula. 
This is because \cref{Thm-WIposmodelcomplete}, the positive model-completeness of $\WSO(\lcro)$, tells us precisely that every $\langWI$-formula is equivalent to some positive existential $\langWI$-formula over $\WSO(\lcro)$.

Using our interpretation of $\WSO(\lcro)$ in $\LCI(\lcro)$ (in particular \cref{LinterpretsW}), for each $\langWI$-formula $\psi(\bar{Y})$ there is an $\langLI$-formula $\theta(\bar{Y})$ such that for each $\bar{B} \in \pow_{\fin}(\lcro)$,
\[
\WSO(\lcro) \models \psi(\bar{B}) \Longleftrightarrow \LCI(\lcro) \models \theta(\bar{B}),
\]
and so in particular taking $\bar{B} = (\bar{A_l},\bar{A_r})$ we get,
\[
\WSO(\lcro) \models \psi(\bar{A_l},\bar{A_r}) \Longleftrightarrow \LCI(\lcro) \models \theta(\bar{A_l},\bar{A_r}). \tag{$\dagger$}\label{dagger}
\]
Moreover, as $\psi(\bar{Y})$ is a positive existential $\langWI$-formula, \cref{LinterpretsW} tells us that $\theta(\bar{Y})$ can additionally be chosen to be an existential $\langLI$-formula.
Note that without the positive model-completeness of $\WSO(\lcro)$, we could not have taken $\psi(\bar{Y})$ positive existential, and therefore could not have taken $\theta$ to be existential.

Combining $\eqref{star}$ and $\eqref{dagger}$ we get that for each $\bar{A} \in \pow_{\fci}(\lcro)$,
\[
\LCI(\lcro) \models \phi(\bar{A}) \Longleftrightarrow \LCI(\lcro) \models \theta(\bar{A_l},\bar{A_r}).
\]
Now, $l$ and $r$ are part of the signature $\langLI$, so we take $\phi^*(\bar{X})$ to be the formula $\theta(\bar{X_l},\bar{X_r})$. 
Our choice of $\phi^*$ is clearly existential, as $\theta$ is existential. 
Putting everything together, we get that,
\[
\LCI(\lcro) \models \forall \bar{X} (\phi(\bar{X}) \leftrightarrow \phi^*(\bar{X})),
\]
and hence $\LCI(\lcro)$ is model-complete. 
\end{proof}

\printbibliography

@book{DLPhD,
    AUTHOR = {Linkhorn, Deacon V.},
     TITLE = {Monadic {S}econd {O}rder {L}ogic and {L}inear {O}rders},
      NOTE = {Thesis (Ph.D.)--The University of Manchester (United Kingdom)},
 PUBLISHER = {ProQuest LLC, Ann Arbor, MI},
      YEAR = {2021},
     PAGES = {148},
      ISBN = {979-8209-93795-1},
   MRCLASS = {Thesis},
  MRNUMBER = {4419961},
       URL =
              {http://gateway.proquest.com.manchester.idm.oclc.org/openurl?url_ver=Z39.88-2004&rft_val_fmt=info:ofi/fmt:kev:mtx:dissertation&res_dat=xri:pqm&rft_dat=xri:pqdiss:29051746},
}

@book {HodgesBible,
    AUTHOR = {Hodges, Wilfrid},
     TITLE = {Model theory},
    SERIES = {Encyclopedia of Mathematics and its Applications},
    VOLUME = {42},
 PUBLISHER = {Cambridge University Press, Cambridge},
      YEAR = {1993},
     PAGES = {xiv+772},
      ISBN = {0-521-30442-3},
   MRCLASS = {03-01 (03-02 03Cxx)},
  MRNUMBER = {1221741},
MRREVIEWER = {J. M. Plotkin},
       DOI = {10.1017/CBO9780511551574},
       URL = {https://doi.org/10.1017/CBO9780511551574},
}

@incollection {TresslTop,
    AUTHOR = {Tressl, Marcus},
     TITLE = {On the strength of some topological lattices},
 BOOKTITLE = {Ordered algebraic structures and related topics},
    SERIES = {Contemp. Math.},
    VOLUME = {697},
     PAGES = {325--347},
 PUBLISHER = {Amer. Math. Soc., Providence, RI},
      YEAR = {2017},
   MRCLASS = {03C64 (03G10 06B99)},
  MRNUMBER = {3716080},
MRREVIEWER = {Tobias Kaiser},
       DOI = {10.1090/conm/697/14060},
       URL = {https://doi.org/10.1090/conm/697/14060},
}

\end{document}